\documentclass[11pt]{amsart}
\usepackage{amsmath,amsthm,amssymb}
\usepackage{eucal}
\usepackage{mathrsfs,nicefrac}
\usepackage{amssymb}
\usepackage[all]{xy}
\usepackage{cancel}
\usepackage{color}
\usepackage{enumerate}

\usepackage[bookmarks=false,pdftex,pdfborder={0 0 0 [1 1]}]{hyperref}

\theoremstyle{plain}
\newtheorem{thm}{Theorem}[section]
\newtheorem*{thm*}{Theorem}
\newtheorem{lem}[thm]{Lemma}
\newtheorem*{lem*}{Lemma}
\newtheorem{prop}[thm]{Proposition}
\newtheorem*{prop*}{Proposition}

\newtheorem*{cor*}{Corollary}

\newtheorem*{conjecture*}{Conjecture}

\theoremstyle{definition}

\newtheorem*{defn*}{Definition}

\renewcommand{\ker}{\textup{ker}\,}

\DeclareMathOperator{\im}{im}

\DeclareMathOperator*{\holim}{holim}
\DeclareMathOperator*{\hofib}{hofib}

\newcommand{\DASH}{\textup{--}}

\let\phi\varphi
\renewcommand{\to}{\longrightarrow}
\newcommand{\from}{\longleftarrow}

\newcommand{\squishlist}{
  \setlength{\itemsep}{.5pt}
  \setlength{\parskip}{0pt}
  \setlength{\parsep}{0pt}}

\newcommand{\fraks}{\mathfrak{s}}
\newcommand{\frakd}{\mathfrak{d}}

\newcommand{\frakb}{\mathfrak{b}}

\newcommand{\calP}{\mathcal{P}}
\newcommand{\calA}{\mathcal{A}}

\newcommand{\calD}{\mathcal{D}}

\newcommand{\calX}{\mathcal{X}}
\newcommand{\calC}{\mathcal{C}}

\usepackage{framed}
\usepackage[style=numeric,
url=false,doi=false,isbn=false,eprint=false]{biblatex}%
\hypersetup{colorlinks=false,pdfborder={0 0 0}}

\makeatletter
\renewcommand{\@seccntformat}[1]{\csname the#1\endcsname.\quad}
\makeatother

\theoremstyle{plain}

\newtheorem*{completenesstheorem}{Completeness Theorem}
\newtheorem{twistinglemma}[thm]{Twisting Lemma}

\newcommand{\Comm}{\calC}

\bibliography{papers}

\title{Connected simplicial algebras are\\Andr\'e-Quillen complete}
\author[M.\ Donovan]{Michael Donovan}

\address{Department of Mathematics \\ Massachusetts Institute of Technology}
\email{mdono@math.mit.edu}

\newcommand{\dupdown}[2]{D_{\smash{#1}}}
\newcommand{\caldup}[1]{\calD_{\smash{#1}}}

\begin{document}

\begin{abstract}
We modify a classical construction of Bousfield and Kan \cite{BousKanSSeq.pdf} to define the Adams tower of a simplicial nonunital commutative algebra over a field $k$. We relate this construction to Radulescu-Banu's cosimplicial resolution \cite{Radulescu-Banu.pdf}, and prove that all connected simplicial algebras are complete with respect to Andr\'e-Quillen homology. This is a convergence result for the unstable Adams spectral sequence for commutative algebras over $k$.
\end{abstract}

\maketitle

Let $s\calC$ denote the simplicial model category \cite{QuillenHomAlg.pdf} of simplicial non-unital commutative algebras over a field $k$, and let $X$ be an object of $s\calC$. Radulescu-Banu \cite{Radulescu-Banu.pdf} constructed a cosimplicial resolution $\calX^\bullet$ of $X$ by generalised Eilenberg-Mac Lane objects, and defined the \emph{completion of $X$ with respect to Andr\'e-Quillen homology} to be the totalization $X\hat{\ }:=\textup{Tot}(\calX^\bullet)$. The purpose of the present work is to prove the following conjecture of Radulescu-Banu:
\begin{completenesstheorem}\label{completenesstheorem}
If $X$ is a connected simplicial $k$-algebra, then $X$ is naturally equivalent to its completion $X\hat{\ }$.
\end{completenesstheorem}
Radulescu-Banu's completion functor $X\mapsto X\hat{\ }$ is the analogue of Bousfield and Kan's $R$-completion functor on simplicial sets \cite{BousKanSSeq.pdf}, a construction that has proven extremely useful in classical homotopy theory. As in the classical case, the homotopy of the completion $X\hat{\ }$ is the target of an unstable Adams spectral sequence. The completeness theorem may be viewed as a convergence result for this spectral sequence, which we will study in detail in forthcoming work.

In fact, Bousfield and Kan have defined the unstable Adams spectral sequence of a simplicial set in two different ways. Their earlier approach \cite{BK_pairings.pdf} was to define the \emph{derivation of a functor with respect to a ring}. This approach constructs the \emph{Adams tower} over the simplicial set in question, and lends itself well to connectivity analyses. Their latter approach, \cite{BousKanSSeq.pdf}, to give a cosimplicial resolution of a simplicial set by simplicial $R$-modules, lends itself more to the analysis of the $E^2$ page, and is directly analogous to Radulescu-Banu's construction.

Since the release of \cite{BK_pairings.pdf} and \cite{BousKanSSeq.pdf}, the relationship between the two approaches has been clarified by the introduction of the theory of cubical diagrams to homotopy theory \cite{GoodwillieCalcII}. Our approach to proving the completeness theorem will be to define the Adams tower of a simplicial algebra using a construction analogous to Bousfield and Kan in \cite{BK_pairings.pdf} (Section \ref{sec:derWRTab}), and to use the theory of cubical diagrams to relate it to Radulescu-Banu's construction (Section \ref{sec:relnWithRB}). In Section \ref{sec:connectivityAnalysis}, we perform the necessary connectivity estimates in the Adams tower in order to prove the completeness theorem. 

As Radulescu-Banu observed, there is an additional difficulty in constructing the Bousfield-Kan cosimplicial resolution of a simplicial algebra which is not present in the classical context. Namely, since not all simplicial algebras are cofibrant, the naive cosimplicial resolution will not be homotopically correct. Radulescu-Banu's innovation was to explain that the cofibrant replacement functor $c:s\calC\to s\calC$ constructed by Quillen's small object argument \cite{QuillenHomAlg.pdf} admits a comonad diagonal $\psi:c\to cc$, and can thus be mixed into the cosimplicial resolution, making it homotopically correct. 

Accordingly, we must mix Quillen's cofibrant replacement functor $c$ into our definition of the Adams tower over a simplicial algebra so that it relates as desired to Radulescu-Banu's resolution. The application of these cofibrant replacement functors adds to the difficulty of proving the connectivity estimates of Section \ref{sec:connectivityAnalysis}. We circumvent this difficulty by shifting to the standard comonadic bar construction.

In Section \ref{introToRBwork} we will introduce Radulescu-Banu's cosimplicial resolution and explain a little terminology. In the appendix we will state and prove a useful result on iterated simplicial bar constructions.

This research will form part of the author's PhD thesis. The author would like to thank his advisor, Haynes Miller, for his support and guidance, and John Harper, for sharing helpful insights into the use of cubical diagrams in connectivity analyses such as that of Section \ref{sec:connectivityAnalysis}.

\section{Radulescu-Banu's completion functor}\label{introToRBwork}
A non-unital commutative $k$-algebra is a $k$-vector space $Y$ equipped with an associative and commutative map $\mu:Y\otimes Y\to Y$. We will refer to such objects simply as \emph{algebras}. It can be shown that if $Y\in\calC$ is a categorical group object, then it is in fact a zero-square algebra, that is  $Y^2=\im(\mu)=0$, and the group map $Y\times Y\to Y$ is simply the vector space addition. Thus, the abelianization adjunction for $\calC$ can be modeled as $Q:s\calC\rightleftarrows s\mathsf{Vect}:K$, in which the left adjoint is the \emph{abelianization} or \emph{indecomposables} functor
\[Q:\calC\to \mathsf{Vect},\qquad Y\mapsto Y/Y^2,\]
and the right adjoint is the \emph{zero-square} functor
\[K:\mathsf{Vect}\to\calC,\qquad V\mapsto V\textup{ with $V^2$ set to zero}.\]
The Andr\'e-Quillen homology of a simplicial algebra $X$ is defined as the homotopy groups of its left derived abelianization:
\[H_*X:=\pi_*(QcX),\]
where $c:s\calC\to s\calC$ is the cofibrant replacement from Quillen's SOA. For an excellent introduction to these ideas, see \cite[\S4]{MR1089001}. Radulescu-Banu constructed \cite{Radulescu-Banu.pdf} a comonad diagonal $\psi:c\to cc$, in order to define the coface maps in a (coaugmented) cosimplicial object:
\[\makebox[0cm][r]{\,$\calX^\bullet:\qquad $}\vcenter{
\def\labelstyle{\scriptstyle}
\xymatrix@C=1.5cm@1{
cX\,
\ar[r]
&
\,cKQcX\,
\ar[r];[]
&
\,c(KQc)^2X\,
\ar@<-1ex>[l];[]
\ar@<+1ex>[l];[]
\ar@<+1ex>[r];[]
\ar@<-1ex>[r];[]
&
\,c(KQc)^3X\,\makebox[0cm][l]{\,$\cdots. $}
\ar[l];[]
\ar@<-2ex>[l];[]
\ar@<+2ex>[l];[]
}}\]
The construction is explained and generalized by Blumberg and Riehl \cite{BlumRiehlResolutions.pdf}. The definition of the coface and codegeneracy maps is similar to that in the monadic resolution of $X$ using the adjunction $Q\dashv K$, however, the coface maps must create an extra copy of $c$, and the codegeneracies must destroy a copy.  Instead of simply using the unit and counit of the adjunction respectively, one uses the composites
\[c\overset{\psi}{\to}cc\overset{c\eta c}{\to}cKQc\textup{\quad and\quad }QcK\overset{Q\epsilon K}{\to}QK\to \textup{id}.\]

Denote by $X\hat{\ }$ the totalization of $\calX^\bullet$, naming $X\hat{\ }$ the \emph{completion of $X$ with respect to Andr\'e-Quillen homology}. There is a natural zig-zag $X\overset{\sim}{\from} cX\to X\hat{\ }$, and one says that $X$ is \emph{complete with respect to Andr\'e-Quillen homology} when the map $cX\to X\hat{\ }$ is an equivalence. We will prove that this occurs whenever $X$ is connected. In order to understand the name of the completion functor, note that levelwise application of Andr\'e-Quillen homology to $\calX^\bullet$ yields a cosimplicial vector space $H_*\calX^\bullet$ which is weakly equivalent to its coaugmentation $H_*X$ (cf. \cite{BlumRiehlResolutions.pdf}).

\section{The Adams tower}\label{sec:derWRTab}

For any functor $F:s\calC\to s\calC$, we define the $r^\textup{th}$ derivation $\dupdown{r}{b}F$ of $F$ with respect to Andr\'e-Quillen homology. The definition is recursive:
\begin{alignat*}{2}
(\dupdown{0}{b}F)(X)
&:=
F(cX)%
\\
(\dupdown{s}{b}F)(X)
&:=
\hofib((\dupdown{s-1}{c}F)(cX)\xrightarrow{(\dupdown{s-1}{c}F)(\eta_{cX})} (\dupdown{s-1}{c}F)(KQcX))
\end{alignat*}
where $\eta$ is the unit of the adjunction $Q\dashv K$, i.e.\ the natural surjection onto indecomposables, and $\hofib$ is any fixed  functorial construction of the homotopy fiber. These functors fit into a tower via the following composite natural transformation:
\[\delta:\left((\dupdown{s}{c}F)(X)\to (\dupdown{s-1}{c}F)(cX)\overset{(\dupdown{s-1}{c}F)(\epsilon)}{\to} (\dupdown{s-1}{c}F)(X)\right).\]
We have thus constructed a tower
\[\xymatrix@R=4mm{
\cdots 
\ar[r]
&
(\dupdown{2}{c}F)X
\ar[r]
&
(\dupdown{1}{c}F)X
\ar[r]&
(\dupdown{0}{c}F)X=FcX,
}\]
which is natural in the object $X$ and the functor $F$.
The functors $\dupdown{r}{c}F$ are homotopical as long as $F$ preserves weak equivalences between cofibrant objects. Employing the shorthand
\[\dupdown{s}{c}X:=(\dupdown{s}{c}I)X,\]
we define \emph{the Adams tower of $X$} to be the tower
\[\xymatrix@R=4mm{
\cdots 
\ar[r]
&
\dupdown{2}{c}X
\ar[r]
&
\dupdown{1}{c}X
\ar[r]&
\dupdown{0}{c}X=cX.
}\]

For example, $(\dupdown{2}{c}F)(X)$ is constructed by the following diagram in which every composable pair of parallel arrows is \emph{defined} to be a homotopy fiber sequence.
\[\def\labelstyle{\scriptstyle}
\xymatrix@!0@R=30pt@C=43pt{
(\dupdown{2}{c}F)(X)\ar[d]\\
(\dupdown{1}{c}F)(cX) \ar[rr]\ar[d]         &           &FcccX \ar[rr]\ar[d]         &           &   FcKQccX            \ar[d]  &                  \\
(\dupdown{1}{c}F)(KQcX) \ar[rr] &                     &  FccKQcX \ar[rr] &             & FcKQcKQcX
}\]
In general, $(\dupdown{n+1}{c}F)(X)$ is the homotopy total fiber of an $(n+1)$-cubical diagram:
\[(\dupdown{n+1}{c}F)(X):=\textup{hototfib} \bigl(({D}_{n+1}^{\smash{\square}}F)X\bigr).\]
See \cite{GoodwillieCalcII}, \cite{LuisGoodwillie.pdf} or \cite{CubicalHomotopyTheory.pdf} for the general theory of cubical diagrams. Before defining the cubical diagram $({D}_{n+1}^{\smash{\square}}F)X$, we set notation: for $n\geq0$ let $[n]=\{0,\ldots,n\}$, and define $\calP[n]=\left\{S\subseteq [n]\right\}$ to be the poset category whose morphisms are the inclusions $S\subseteq S'$. Then an $(n+1)$-cube in $s\calC$ is a functor $\calP[n]\to s\calC$, and the $(n+1)$-cubical diagram $({D}_{n+1}^{\smash{\square}}F)X:\calP[n]\to s\calC$ is the functor:
\[S\mapsto Fc(KQ)^{\chi_{n}}c(KQ)^{\chi_{n-1}}c\cdots c(KQ)^{\chi_0}cX\quad \textup{where}\quad \chi_i:=\begin{cases}
1,&\textup{if }i\in S;\\
0,&\textup{if }i\notin S,
\end{cases}
\]
where for $S\subseteq S'$, the map $(({D}_{n+1}^{\smash{\square}}F)X)(S)\to (({D}_{n+1}^{\smash{\square}}F)X)(S')$ is given by applying the counit $\eta:1\to KQ$ in those locations indexed by $S'\setminus S$.

\section{Relationship between the Adams tower and Radulescu-Banu's resolution}\label{sec:relnWithRB}

Radulescu-Banu defines the completion of $X$ to be the totalization
\[X\hat{\ }:=\textup{Tot}(\calX^\bullet)=\holim (\textup{Tot}_n(\calX^\bullet)),\]
and the unstable Adams spectral sequence to be the spectral sequence of the tower
\[\cdots \to\textup{Tot}_n(\calX^\bullet)\to \textup{Tot}_{n-1}(\calX^\bullet)\to\cdots \]
under $cX$. Our goal in this section is to prove
\begin{prop}\label{towerIdentification}
There is a natural zig-zag of weak equivalences of towers between $D_{n+1}X$ and $\hofib(cX\to\textup{Tot}_{n}(\calX^\bullet))$. That is, up to homotopy, the $\textup{Tot}$ tower induces the Adams tower by taking fibers.
\end{prop}
Before giving the proof, we recall a useful relationship between cosimplicial objects and cubical diagrams, explained by Sinha in \cite[Theorem 6.5]{SinhaSpacesOfKnots.pdf}, and expanded on by Munson-Voli\'c \cite{CubicalHomotopyTheory.pdf}.
We will only present that part of the theory that we need. There is a diagram of inclusions of categories
\[\xymatrix@!0@R=10mm@C=13mm{
\calP[-1]\ar[rr]^-{\tau}\ar_-{h_{-1}}[drrr]
&&
\calP[0]\ar[rr]^-{\tau}\ar^(.65){\!h_{0}}[dr]
&&
\calP[1]\ar[rr]^-{\tau}\ar_(.65){h_{1}\!}[dl]
&&
\calP[2]\ar[rr]^-{\tau}\ar^-{h_{2}}[dlll]
&&
\cdots \\
&&&\Delta_+\!\!
}\]
As the coaugmented cosimplicial object $\calX^\bullet$ is Reedy fibrant (cf.\ \cite[{X.4.9}]{YellowMonster}), there are natural weak equivalences $\hofib(\calX^{-1}\to\textup{Tot}_n\calX^\bullet) \overset{\smash{\sim}}{\to} \textup{hototfib}(h_n^*(\calX^\bullet))$ under which the tower map 
\[\hofib(\calX^{-1}\to\textup{Tot}_n\calX^\bullet)\to \hofib(\calX^{-1}\to\textup{Tot}_{n-1}\calX^\bullet)\]
is identified with the map
\[\textup{hototfib}(h_n^*(\calX^\bullet))\to \textup{hototfib}(\tau^*h_n^*(\calX^\bullet))=\textup{hototfib}(h_{n-1}^*(\calX^\bullet)).\]
As $\calX^{-1}$ equals $cX$, the tower $\hofib(\calX^{-1}\to\textup{Tot}_n\calX^{\bullet})$ is one of the towers in proposition \ref{towerIdentification}.
\begin{proof}[Proof of Proposition \ref{towerIdentification}] 
It will suffice to construct a weak equivalence $h_n^*\calX^\bullet\to (D_n^{\smash{\square}}I)(X)$ of $(n+1)$-cubes. The $(n+1)$-cubical diagram $h_n^*\calX^\bullet$ is defined by
\[(h_n^*\calX^\bullet)(S):= c(KQc)^{\chi_{n}}(KQc)^{\chi_{n-1}}\cdots (KQc)^{\chi_0}X\quad \textup{where}\quad \chi_i:=\begin{cases}
1,&\textup{if }i\in S,\\
0,&\textup{if }i\notin S,
\end{cases}
\]
where we describe the map $(h_n^*\calX^\bullet)(S)\to (h_n^*\calX^\bullet)(S\sqcup\{i\})$, for $i\notin S$, as follows. Let $j$ be the smallest element of $S\sqcup\{n+1\}$ exceeding $i$, so that
\[(h_n^*\calX^\bullet)(S):= \begin{cases}
c(KQc)^{\chi_{n}}\cdots (KQc)^{\chi_{j+1}}(KQ\underline{c})(KQc)^{\chi_{i-1}}\cdots (KQc)^{\chi_0}X,&\textup{if }j\leq n;\\
\underline{c}(KQc)^{\chi_{i-1}}\cdots (KQc)^{\chi_0}X,&\textup{if }j=n+1.
\end{cases}
\]
In the expression for either case, we have distinguished one of the applications of $c$ with an underline, and the map to $(h_n^*\calX^\bullet)(S\sqcup\{i\})$ is induced by the composite $\underline{c}\to cc\to cKQc$ of the diagonal of the comonad $c$ with the unit of the monad $KQ$. 

We now define maps $(h_n^*\calX^\bullet)(S)\to (({D}_{n+1}^{\smash{\square}}I)X)(S)$ for $S=\{j_0<j_1<\cdots<j_r\}\subset\{0,\ldots,n\}$. 
The only difference between the domain and codomain is that in $(({D}_{n+1}^{\smash{\square}}I)X)(S)$, all $n+2$ applications of $c$ are present, whereas in $(h_n^*\calX^\bullet)(S)$, only $r+2$ appear. The map is then
\[\psi^{n-j_r}KQ\psi^{j_r-j_{r-1}-1}KQ\psi^{j_{r-1}-j_{r-2}-1}KQ\cdots KQ\psi^{j_{1}-j_0-1}KQ\psi^{j_0}X\]
which is to say that we apply the iterated diagonal the appropriate number of times in each $c$ appearing in the domain. As $\psi$ is coassociative, this definition is unambiguous, and the resulting maps assemble to a weak equivalence of $(n+1)$-cubes. 
\end{proof}

\section{Connectivity estimates in the Adams tower}
\label{sec:connectivityAnalysis}
We wish to prove the following connectivity result for the Adams tower:
\begin{prop}\label{convergenceProp}
For integers $t\geq2$ and $q\geq0$, the map $\pi_q(\dupdown{2t+q-1}{c}X)\to\pi_q(\dupdown{t}{c}X)$ is zero.
\end{prop}
\noindent We defer the proof until the end of this section. Note that proposition \ref{towerIdentification} and \ref{convergenceProp} together imply the completeness theorem:
\begin{proof}[Proof of completeness theorem]
The fiber sequences $\dupdown{n+1}{c}X\to cX\to \textup{Tot}_n\calX^\bullet$ fit together into a tower of fiber sequences. Taking homotopy limits, one obtains a fiber sequence
\[\holim (\dupdown{n}{c}X)\to cX\to X\hat{\ }.\]
We need to show that $\holim (\dupdown{n}{c}X)$ has zero homotopy groups.
We may replace the tower $\dupdown{n}{c}X$ with a weakly equivalent tower of fibrations in $s\calC$ whose set-theoretic inverse limit is the homotopy limit in question. Applying \cite[Proposition 6.14]{goerss-jardine.pdf} to the new tower, there is a short exact sequence
\[0\to \textup{lim}^1\pi_{q+1}(\dupdown{n}{c}X)\to \pi_q(\holim (\dupdown{n}{c}X))\to \textup{lim}\,\pi_{q}(\dupdown{n}{c}X)\to 0.\]
Proposition \ref{convergenceProp} implies that for each $q$, the tower $\{\pi_{q}(\dupdown{n}{c}X)\}$ has zero inverse limit, and satisfies the Mittag-Leffler condition (cf.\ \cite[p.264]{YellowMonster}), so that the $\textup{lim}^1$ groups appearing also vanish.
\end{proof}
With an eye to proving proposition \ref{convergenceProp} we define a somewhat less homotopical version $\caldup{s}F$ of the derivations $\dupdown{s}{b}F$ of $F$. Again, the definition is recursive:
\begin{alignat*}{2}
(\caldup{0}F)(X)
&:=
F(X)%
\\
(\caldup{s}F)(X)
&:=
\ker((\caldup{s-1}F)(bX)\xrightarrow{(\caldup{s-1}F)(\eta_{bX})} (\caldup{s-1}F)(KQbX))
\end{alignat*}
There are three differences between this definition and that of $D_sF$: here, there is one fewer cofibrant replacement applied, we use the standard simplicial bar construction $b$ (see Appendix \ref{sec:ItSimpBar}) instead of the SOA functor $c$, and we take \emph{strict} fibers, not \emph{homotopy} fibers.
While these functors are not generally homotopical, we define \emph{the modified Adams tower of $X$} to be the tower
\[\xymatrix@R=4mm{
\cdots 
\ar[r]
&
\caldup{2}X
\ar[r]
&
\caldup{1}X
\ar[r]&
\caldup{0}X=X,
}\]
where $\caldup{s}X$ is again shorthand for $(\caldup{s}I)X$, and the tower maps $\delta$ are defined as before.
\begin{prop}\label{prop:modifiedAdamsTower}
There's a natural zig-zag of weak equivalences of towers between the Adams tower of $X$ and the modified Adams tower of $X$. In particular, the modified tower is homotopical.
\end{prop}
\begin{proof}
Let $\mathsf{CR(s\calC)}$ be the category of cofibrant replacement functors in $s\calC$. That is, an object of $\mathsf{CR(s\calC)}$ is a pair, $(f,\epsilon)$, such that $f:s\calC\to s\calC$ is a functor whose image consists only of cofibrant objects, and $\epsilon:f\Rightarrow I$ is a natural acyclic fibration. Morphisms in $\mathsf{CR(s\calC)}$ are natural transformations which commute with the augmentations. For any $(f,\epsilon)\in\mathsf{CR(s\calC)}$ we obtain an alternative definition of the derivations of a functor $F:s\calC\to s\calC$:
\begin{alignat*}{2}
(D_{\smash{0}}^{\smash{f}}F)(X)
:=
F(fX),\quad %
(D_{\smash{s}}^{\smash{f}}F)(X)
:=
\hofib((D_{\smash{s-1}}^{\smash{f}}F)(fX)\to (D_{\smash{s-1}}^{\smash{f}}F)(KQfX)).
\end{alignat*}
These functors are natural in $f$, so that a morphism in $\mathsf{CR(s\calC)}$ induces a weak equivalence of towers. Our proposed zig-zag of towers is: 
\[D_{\smash{s}}= D_{\smash{s}}^{\smash{c}}I\overset{}{\from}D_{\smash{s}}^{\smash{b\circ c}}\!I\overset{}{\to}D_{\smash{s}}^{\smash{b}}I\overset{\gamma_s}{\from}\caldup{s}b\overset{\caldup{s}\epsilon}{\to}\caldup{s}I=\caldup{s}\]
The maps with domain $D_{\smash{s}}^{\smash{b\circ c}}I$ are induced by the maps $\epsilon c:b\circ c\to c$ and $b\epsilon:b\circ c\to b$ and are evidently natural weak equivalences of towers. The map $\gamma_0:(\caldup{0}b)X\to (D_{\smash{0}}^{\smash{b}}I)X$ is the identity of $bX$, and the map $\caldup{0}\epsilon:(\caldup{0}b)X\to (D_{\smash{0}}^{\smash{b}}I)X$ is $\epsilon:bX\to X$. Thereafter, $\gamma_s$ and $\caldup{s}\epsilon$ are defined recursively:
\[\qquad \quad \xymatrix@R=4mm{
\makebox[0cm][r]{$(\caldup{s+1}I)X:=\,$}\makebox[5cm][l]{$\,\ker((\caldup{s}I)(bX)\to (\caldup{s}I)(KQbX))$}\\
\makebox[0cm][r]{$(\caldup{s+1}b)X:=\,$}\makebox[5cm][l]{$\,\ker((\caldup{s}b)(bX)\to (\caldup{s}b)(KQbX))$}\ar[d]^-{\textup{incl.}}_-{}
\ar[u]_-{\textup{induced by }(\caldup{s}\epsilon,\caldup{s}\epsilon)}^-{}
\\
\makebox[5cm][l]{$\,\hofib((\caldup{s}b)(bX)\to(\caldup{s}b)(KQbX))$}\ar[d]^-{\textup{induced by }(\gamma_{s},\gamma_{s})}_-{}
\\
\makebox[0cm][r]{$(D_{\smash{s+1}}^{\smash{b}})X:=\,$}\makebox[5cm][l]{$\,\hofib((D_{\smash{s}}^{\smash{b}}I)(bX)\to (D_{\smash{s}}^{\smash{b}}I)(KQbX))$}
}\]
\noindent Lemma \ref{towerWithPowers} shows that the kernels taken are actually kernels of surjective maps, and by induction on $s$, the maps $\gamma_s$ and $\caldup{s}\epsilon$ are weak equivalences.
\end{proof}
The connectivity result will rely on the observation that any element in the $s^\textup{th}$ level of the modified tower maps down to an $(s+1)$-fold product in $X$. In order to formalise this, let $P^s:s\calC\to s\calC$ be the ``$s^\textup{th}$ power'' functor, the prolongation of the endofunctor $Y\mapsto Y^s$ of $\calC$, where $Y^s=\textup{im}(\textup{mult}:Y^{\otimes s}\to Y)$. Then we have:
\begin{lem}\label{towerWithPowers}
The functors $\caldup{r}$, $\caldup{r}b$ and $\caldup{r}P^{s}$ preserve surjective maps and there is a commuting diagram of functors:
\[\xymatrix@R=4mm{
\cdots 
\ar[r]
&
\caldup{r}
\ar[r]
\ar[d]
&
\cdots \ar[r]
&
\caldup{2}
\ar[r]
\ar[d]
&
\caldup{1}
\ar[r]
\ar[d]&
\caldup{0}
\ar@{=}[d]
\\
\cdots
\ar[r]
&
P^{r+1}
\ar[r]
&
\cdots 
\ar[r]&
P^3
\ar[r]
&P^2
\ar[r]
&
I
}\]
\end{lem}
\begin{proof}

Suppose that $X$ is a simplicial algebra. Then $\caldup{r}X$
is constructed as the subalgebra
\[\caldup{r}X:= \bigcap_{i=1}^{r}\ker\!\left(b^{r-i}\eta b^{i}:b^{r}X\to b^{r-i}KQb^{i}X\right)\]
of $b^rX$. In dimension $n$, this is the following subset of $(b^rX)_n:=(T^{n+1})^rX_n$:
\[(\caldup{r}X)_n:=\bigcap_{i=1}^{r}\ker\!\left(T^{(r-i)(n+1)}\eta T^{i(n+1)}:(T^{n+1})^rX_n\to (T^{n+1})^{r-i}KQ(T^{n+1})^iX_n\right)\]
As in the appendix, $T$ is the free tensor algebra comonad on $\calC$.
The $r$ conditions on elements of $(\caldup{r}X)_n$ ensure that their image in $(\caldup{0}X)_n=X_n$ under the iterated tower map is a sum of $(r+1)$-fold products. This completes the construction of the tower of functors.

In order to prove the surjectivity statements, we must describe the iterated free construction $T^{r(n+1)}X_n$. A basis of $TX_n$ may be given by the \emph{monomials} in a basis of $X_n$, and $T^{r(n+1)}X_n$ has basis given by taking monomials iteratively, $r(n+1)$ times. The subset $(\caldup{r}X)_n$ has basis those iterated monomials in which the monomials formed in the $((n+1)i)^\textup{th}$ iteration have degree at least two for each $1\leq i\leq r$. This simple description of a basis of $(\caldup{r}X)_n$ shows that $\caldup{r}$ preserves surjections. Similar analysis applies to $\caldup{r}b$ and $\caldup{r}P^s$.
\end{proof}
We are now able to state and prove the key connectivity result:
\begin{lem}\label{connectivityOfDerivedPowers}
For $A\in s\Comm$ connected, and any $t\geq1$ and $s\geq2$, $(\caldup{t}P^{s})(A)$ is $(s-t)$-connected.
\end{lem}
\begin{proof}
We will prove this by induction on $t$. When $t=1$:
\[(\caldup{1}P^{s})A:=\ker(P^{s}(bA)\to P^{s}(QbA))=P^{s}(bA),\]
so we must show that $P^s(bA)$ is $(s-1)$-connected. For this we use a truncation of Quillen's fundamental spectral sequence, as presented in \cite[Thm 6.2]{MR1089001}: the filtration
\[P^s(bA)\supset P^{s+1}(bA)\supset P^{s+2}(bA)\supset\cdots \]
of $P^s(bA)$ yields a convergent spectral sequence  $E^0_{p,q}\implies \pi_q(P^s(ba))$, with:
\[E^0_{p,q}=N_q\bigl(P^{p}(bA)/P^{p+1}(bA)\bigr)\textup{ if $p\geq s$, and }E^0_{p,q}=0\textup{ if $p<s$.}\]
Examination of the structure of $bA$ reveals that $P^{p}(bA)/P^{p+1}(bA)\cong (QbA)^{\otimes p}_{\Sigma_p}$, and moreover, $QbA$ is a connected simplicial vector space, as $\pi_0(QbA)=Q(\pi_0bA)=Q(0)=0$. The $t=1$ case now follows from \cite[Satz 12.1]{DoldPuppeSuspension.pdf}: if $V$ is a connected simplicial vector space then $V^{\otimes p}_{\Sigma_p}$ is $(p-1)$-connected. 

Now let $t\geq2$ and suppose by induction that $(\caldup{t-1}P^{s})(B)$ is $(s-(t-1))$-connected for any connected $B$ and any $s\geq2$. Then by lemma \ref{towerWithPowers}, there's a short exact sequence:
\[\xymatrix{
0\ar[r]&
(\caldup{t}P^{s})(A)\ar[r]&
(\caldup{t-1}P^{s})(bA)\ar[r]&
(\caldup{t-1}P^{s})(QbA)\ar[r]&
0
}\]
in which the rightmost two objects are each $(s-t+1)$-connected. The associated long exact sequence shows that $(\caldup{t}P^{s})(A)$ is $(s-t)$-connected.
\end{proof}
Before we can give the proof of proposition \ref{convergenceProp}, we need the following \emph{twisting lemma}, analogous to that of \cite{BK_pairings.pdf}. Before stating it, we note that $(\caldup{s}\caldup{t})X$ and $\caldup{s+t}X$ are equal by construction.
\begin{twistinglemma}
\label{DsDt=Dt+s}
The maps $\caldup{i}\delta:\caldup{n}X\to \caldup{n-1}X$ are homotopic for $0\leq i< n$.
\end{twistinglemma}
\begin{proof}
We may reindex the twisting lemma as follows: the maps 
\[\caldup{s}\delta,\caldup{s-1}\delta:\caldup{s+t}X\to \caldup{s+t-1}X\]
are homotopic whenever $s,t\geq1$. Now $\caldup{s+t}X$ is constructed as the subalgebra
\[\caldup{s+t}X:= \bigcap_{i=1}^{s+t}\ker\!\left(b^{s+t-i}\eta b^{i}:b^{s+t}X\to b^{s+t-i}KQb^{i}X\right)\]
of the iterated bar construction $b^{s+t}X$, and for $0\leq i<s+t$, $\caldup{i}\delta$ is the restriction of the map $b^i\epsilon b^{s+t-i-1}:b^{s+t}X\to b^{s+t-1}X$.
Proposition \ref{IteratedBarConstructionHomotopy} gives an explicit simplicial homotopy between the maps $b^s\epsilon b^{t-1}$ and $b^{s-1}\epsilon b^{t}$. Moreover, the naturality of the construction of proposition \ref{IteratedBarConstructionHomotopy} implies that this homotopy does indeed restrict to a homotopy of maps $\caldup{s+t}X\to \caldup{s+t-1}X$.
\end{proof}

Now that we have the twisting lemma, Proposition \ref{convergenceProp} follows:
\begin{proof}[Proof of Proposition \ref{convergenceProp}]
By proposition \ref{prop:modifiedAdamsTower}, it is enough to prove that $\pi_q(\caldup{2t+q-1}X)\to\pi_q(\caldup{t}X)$ is zero.
Apply $\caldup{t}\DASH$ to the diagram of functors constructed in \ref{towerWithPowers} and apply the result to $X$ to obtain a commuting diagram of functors
\[\xymatrix@R=4mm{
\caldup{2t+q-1}X
\ar[r]^-{\caldup{t}\delta}
\ar[d]
&
\cdots \ar[r]^-{\caldup{t}\delta}
&
\caldup{t+1}X
\ar[r]^-{\caldup{t}\delta}
\ar[d]&
\caldup{t}X
\ar@{=}[d]
\\
\caldup{t}P^{t+q}X
\ar[r]
&
\cdots 
\ar[r]&
\caldup{t}P^2X
\ar[r]
&
\caldup{t}P^1X
}\]
By the twisting lemma, \ref{DsDt=Dt+s}, the composite along the top row is homotopic to the map of interest, and factors through $(\caldup{t}P^{t+q})(A)$, which is $q$-connected by lemma \ref{connectivityOfDerivedPowers}.
\end{proof}

\appendix
\section{Iterated simplicial bar constructions}\label{sec:ItSimpBar}
\newcommand{\algCat}{\calC}
\newcommand{\trip}[3]{{#1}_{\smash{#2}}^{\smash{#3}}}
In this section we will be concerned with iterated simplicial bar constructions. The result here applies in general in the category of algebras over a monad, but nonetheless we only state it for commutative algebras. Establishing notation, for any simplicial object $X$ in $\algCat$, we'll write 
\[\trip{d}{i,q}{X}:X_q\to X_{q-1}\textup{\ and\ }\trip{s}{i,q}{X}:X_q\to X_{q+1}\]
for the $i^\textup{th}$ face and degeneracy maps out of $X_q$. Suppose $F,G\in\algCat^\algCat$ are endofunctors, $\Phi:F\to G$ is a natural transformation, and $A,B\in\algCat$ are objects. Write $[\Phi]:\algCat(A,B)\to\algCat(FA,GB)$ for the operator sending $m:A\to B$ to the diagonal composite in the commuting square
\[\xymatrix@R=4mm{
FA
\ar[r]^-{\Phi_A}
\ar[d]_-{Fm}
\ar[dr]|-{[\Phi]m}
&
GA
\ar[d]^-{Gm}
\\
FB
\ar[r]_-{\Phi_B}
&
GB
}\]
Write $T:\algCat\to \algCat$ for the comonad of the adjunction $\textup{free}:\algCat\rightleftarrows\mathsf{Vect}:\textup{forget}$. Then there is an (augmented) simplicial endofunctor, $\frakb\in s\algCat^\algCat$, derived from the unit and counit of the adjunction:
\[\vcenter{
\def\labelstyle{\scriptstyle}
\xymatrix@C=1.65cm@1{
{\ I\,}
&
\,T^1\,
\ar@{..>}[l]|(.65){\frakd_{0,0}}
\ar[r]|(.65){\fraks_{0,0}}
&
\,T^2\,
\ar@<-1ex>[l]|(.65){\frakd_{0,1}}
\ar@<+1ex>[l]|(.65){\frakd_{1,1}}
\ar@<+1ex>[r]|(.65){\fraks_{0,1}}
\ar@<-1ex>[r]|(.65){\fraks_{1,1}}
&
\,T^3\,
\ar[l]|(.65){\frakd_{1,2}}
\ar@<-2ex>[l]|(.65){\frakd_{0,2}}
\ar@<+2ex>[l]|(.65){\frakd_{2,2}}
\ar[r]|(.65){\fraks_{1,2}}
\ar@<+2ex>[r]|(.65){\fraks_{0,2}}
\ar@<-2ex>[r]|(.65){\fraks_{2,2}}
&
\,T^4\,\makebox[0cm][l]{\,$\cdots $}
\ar@<-3ex>[l]|(.65){\frakd_{0,3}}
\ar@<-1ex>[l]|(.65){\frakd_{1,3}}
\ar@<+1ex>[l]|(.65){\frakd_{2,3}}
\ar@<+3ex>[l]|(.65){\frakd_{3,3}}
}}\]
The \emph{simplicial bar construction} is the cofibrant replacement functor $(b,\epsilon)$ on $s\algCat$ which is the diagonal of the bisimplicial object obtained by application of $\frakb$ levelwise. That is, for $X\in s\algCat$, $bX$ is the simplicial object with $(bX)_q:=T^{q+1}X_q$, and with
\[\trip{d}{i,q}{bX}:=[\frakd_{i,q}]\trip{d}{i,q}{X}.\]
The augmentation $\epsilon:b\to I$ is defined on level $q$ by 
\[\epsilon_q=\frakd_{0,0}\frakd_{0,1}\cdots \frakd_{0,q}:T^{q+1}\to I\]
We can now construct the simplicial homotopy needed for the twisting lemma, \ref{DsDt=Dt+s}.
\begin{prop}\label{IteratedBarConstructionHomotopy}
The natural transformations $\epsilon_b$ and $b\epsilon$ from $b^2:s\calA\to s\calA$ to $b:s\calA\to s\calA$ are naturally simplicially homotopic.
\end{prop}

\begin{proof}
Write $K=b^{2}X$ and $L=bX$ for the source and target of these maps respectively. Noting the formulae
\[[\frakd_{iq}]^2= [\frakd_{q+i,2q}\circ\frakd_{i,2q+1}]\ \textup{and}\  [\fraks_{iq}]^2= [\fraks_{q+i+2,2q+2}\circ\fraks_{i,2q+1}],\]
we can describe the simplicial structure maps in $K$ and $L$ as follows:
\begin{alignat*}{2}
\trip{d}{iq}{L}&=[\frakd_{iq}]\trip{d}{iq}{X}\\
\trip{s}{iq}{L}&=[\fraks_{iq}]\trip{s}{iq}{X}\\
\trip{d}{iq}{K}&=[\frakd_{q+i,2q}\circ\frakd_{i,2q+1}]\trip{d}{iq}{X}\\
\trip{s}{iq}{K}&=[\fraks_{q+i+2,2q+2}\circ\fraks_{i,2q+1}]\trip{s}{iq}{X}
\end{alignat*}
We can now state an explicit simplicial homotopy between the two maps of interest. Using precisely the notation of \cite[\S5]{MaySimpObj.pdf}, we define $\trip{h}{jq}{}:K_q\to L_{q+1}$, for $0\leq j\leq q$, by the formula
\[\trip{h}{jq}{}:=[\frakd_{j+1,q+2}\circ\cdots \circ\frakd_{j+1,2q+1}]\trip{s}{jq}{X}.\]
We first check that these maps satisfy the defining identities for the notion of simplicial homotopy, numbered (1)-(5) as in \cite[\S5]{MaySimpObj.pdf}. Each identity can be checked in two parts (a)-(b):
{\renewcommand{\circ}{\relax}
\begin{enumerate}\squishlist
\setlength{\parindent}{.25in}
\item We must check that $\trip{d}{i,q+1}{L}\circ \trip{h}{j,q}{}=\trip{h}{j-1,q-1}{}\circ \trip{d}{i,q}{K}$ whenever $0\leq i<j\leq q$, i.e.:
\begin{enumerate}\squishlist
\setlength{\parindent}{.25in}
\item[({\makebox[.51em][c]{a}})] $\trip{d}{i,q+1}{X}\circ \trip{s}{j,q}{X}=\trip{s}{j-1,q-1}{X}\circ \trip{d}{i,q}{X}$,\textup{ and}
\item[({\makebox[.51em][c]{b}})]
$\frakd_{i,q+1}\circ
\frakd_{j+1,q+2}\circ\cdots \circ\frakd_{j+1,2q+1}=
\frakd_{j,q+1}\circ\cdots \circ\frakd_{j,2q-1}\circ
\frakd_{q+i,2q}\circ\frakd_{i,2q+1}$.
\end{enumerate}
\item We must check that $\trip{d}{j+1,q+1}{L}\circ \trip{h}{j,q}{}=\trip{d}{j+1,q+1}{L}\circ \trip{h}{j+1,q}{}$ whenever $0\leq j\leq q-1$, i.e.:
\begin{enumerate}\squishlist
\setlength{\parindent}{.25in}
\item[({\makebox[.51em][c]{a}})] $\trip{d}{j+1,q+1}{X}\circ \trip{s}{j,q}{X}=\trip{d}{j+1,q+1}{X}\circ \trip{s}{j+1,q}{X}$,\textup{ and}
\item[({\makebox[.51em][c]{b}})]
$\frakd_{j+1,q+1}\circ
\frakd_{j+1,q+2}\circ\cdots \circ\frakd_{j+1,2q+1}=
\frakd_{j+1,q+1}\circ
\frakd_{j+2,q+2}\circ\cdots \circ\frakd_{j+2,2q+1}$.
\end{enumerate}
\item We must check that $\trip{d}{i,q+1}{L}\circ \trip{h}{j,q}{}=\trip{h}{j,q-1}{}\circ \trip{d}{i-1,q}{K}$ whenever $0\leq j<i-1\leq q$, i.e.:
\begin{enumerate}\squishlist
\setlength{\parindent}{.25in}
\item[({\makebox[.51em][c]{a}})] $\trip{d}{i,q+1}{X}\circ \trip{s}{j,q}{X}=\trip{s}{j,q-1}{X}\circ \trip{d}{i-1,q}{X}$,\textup{ and}
\item[({\makebox[.51em][c]{b}})] 
$\frakd_{i,q+1}\circ
\frakd_{j+1,q+2}\circ\cdots \circ\frakd_{j+1,2q+1}=
\frakd_{j+1,q+1}\circ\cdots \circ\frakd_{j+1,2q-1}\circ
\frakd_{q+i-1,2q}\circ\frakd_{i-1,2q+1}$.
\end{enumerate}
\item We must check that $\trip{s}{i,q+1}{L}\circ \trip{h}{j,q}{}=\trip{h}{j+1,q+1}{}\circ \trip{s}{i,q}{K}$ whenever $0\leq i\leq j\leq q$, i.e.:
\begin{enumerate}\squishlist
\setlength{\parindent}{.25in}
\item[({\makebox[.51em][c]{a}})] $\trip{s}{i,q+1}{X}\circ \trip{s}{j,q}{X}=\trip{s}{j+1,q+1}{X}\circ \trip{s}{i,q}{X}$,\textup{ and}
\item[({\makebox[.51em][c]{b}})] 
$\fraks_{i,q+1}\circ
\frakd_{j+1,q+2}\circ\cdots \circ\frakd_{j+1,2q+1}=
\frakd_{j+2,q+3}\circ\cdots \circ\frakd_{j+2,2q+3}\circ
\fraks_{q+i+2,2q+2}\circ\fraks_{i,2q+1}$.
\end{enumerate}
\item We must check that $\trip{s}{i,q+1}{L}\circ \trip{h}{j,q}{}=\trip{h}{j,q+1}{}\circ \trip{s}{i-1,q}{K}$ whenever $0\leq j<i\leq q+1$, i.e.:
\begin{enumerate}\squishlist
\setlength{\parindent}{.25in}
\item[({\makebox[.51em][c]{a}})] $\trip{s}{i,q+1}{X}\circ \trip{s}{j,q}{X}=\trip{s}{j,q+1}{X}\circ \trip{s}{i-1,q}{X}$,\textup{ and}
\item[({\makebox[.51em][c]{b}})] 
$\fraks_{i,q+1}\circ
\frakd_{j+1,q+2}\circ\cdots \circ\frakd_{j+1,2q+1}=
\frakd_{j+1,q+3}\circ\cdots \circ\frakd_{j+1,2q+3}\circ
\fraks_{q+i+1,2q+2}\circ\fraks_{i-1,2q+1}$.
\end{enumerate}
\end{enumerate}
\noindent Each of these equations follows from the simplicial identities, proving that the $h_{jq}$ form a homotopy. 
Finally, we check that this homotopy is indeed a homotopy between the two maps of interest:
\begin{alignat*}{2}
\trip{d}{0,q+1}{L}\trip{h}{0,q}{}
&=
[\frakd_{0,q+1}\frakd_{1,q+2}\circ\cdots \circ\frakd_{1,2q+1}](\trip{d}{0,q+1}{X}\trip{s}{0q}{X})%
\\
&=
[\frakd_{0,q+1}\frakd_{0,q+2}\circ\cdots \circ\frakd_{0,2q+1}]\trip{\textup{id}}{X_q}{}%
\end{alignat*}
is the action of  $\epsilon_{(bX)}$ in level $q$, and similarly,
\begin{alignat*}{2}
\trip{d}{q+1,q+1}{L}\trip{h}{q,q}{}
&=
[\frakd_{q+1,q+1}\frakd_{q+1,q+2}\circ\cdots \circ\frakd_{q+1,2q+1}](\trip{d}{q+1,q+1}{X}\trip{s}{qq}{X})
\\
&=
[\frakd_{q+1,q+1}\frakd_{q+1,q+2}\circ\cdots \circ\frakd_{q+1,2q+1}]\trip{\textup{id}}{X_q}{}%
\end{alignat*}
is the action of $b\epsilon_{X}$ in level $q$.
}
\end{proof}

\printbibliography

\end{document}